\documentclass[12pt]{amsart}
\usepackage[utf8]{inputenc}
\usepackage{amsmath,amssymb,amsthm}

\usepackage{thmtools}
\usepackage{thm-restate}

\usepackage{caption} 
\usepackage[en-AU]{datetime2}

\usepackage{graphicx}
\usepackage{hyperref}
\usepackage[dvipsnames]{xcolor}
\usepackage{tikz}
    \usetikzlibrary{calc}
\usepackage{cleveref}
\usepackage{listings}
\newtheorem{theorem}{Theorem}[section]

\newtheorem{lemma}[theorem]{Lemma}

\newtheorem{proposition}[theorem]{Proposition}
\newtheorem{corollary}[theorem]{Corollary}
\newtheorem*{thm*}{Theorem}

\theoremstyle{remark}

\theoremstyle{definition}

\newtheorem{problem}{Problem}

\usepackage{calc}
\usepackage[shortlabels]{enumitem}
    \newlength{\circlabelwidth}
    \setlist{nosep,before={\parskip=0pt plus 2pt},after={\parskip=0.5em plus 2pt}}
    \setlist[enumerate]{label=\textup{\arabic*.}}
    \newlist{subprob}{enumerate}{2}
        \setlist[subprob,1]{label={(\roman*)}}
        \setlist[subprob,2]{label={(\arabic*)}}
    \setlist[itemize]{labelindent=10pt,labelwidth=\circlabelwidth,leftmargin=!,label=$\circ$}
    \newlist{problems}{enumerate}{3}
        \setlist[problems,1]{before=\setupstar,label=\textup{\arabic*.}, itemsep=2pt, topsep=8pt,ref=\textup{\arabic*}}
        \setlist[problems,2]{before=\setupstar,label=(\alph*),parsep=0pt}
        \setlist[problems,3]{before=\setupstar,label=(\roman*),parsep=0pt}

\DeclareMathOperator{\conv}{conv}
\newcommand*{\R}{\mathbb{R}}
\newcommand*{\Z}{\mathbb{Z}}

\title{Improved Helly numbers of product sets}
\author{Srinivas Arun}
\author{Travis Dillon}
\address{Srinivas Arun, Massachusetts Institute of Technology, Cambridge, MA 02139, USA}
\email{sarun@mit.edu}
\address{Travis Dillon, Massachusetts Institute of Technology, Cambridge, MA 02139, USA}
\email{travis.dillon@mit.edu}

\begin{document}

\begin{abstract}
A finite family $\mathcal F$ of convex sets is \emph{$k$-intersecting} in $S \subseteq \R^d$ if the intersection of every subset of $k$ convex sets in $\mathcal F$ contains a point in $S$. The \emph{Helly number} of $S$ is the minimum $k$, if it exists, such that every $k$-intersecting family contains a point of $S$ in its intersection. In this paper, we improve bounds on the Helly number of \emph{product sets} of the form $A^d$ for various sets $A \subseteq \R$, including the ``exponential grid'' $A = \{\alpha^n : n \in \mathbb{N}\}$ and sets $A\subseteq \mathbb{Z}$ defined by congruence relations.
\end{abstract}

\maketitle

\section{Introduction}

Helly's theorem is a fundamental result in discrete and convex geometry. It says that \emph{in a finite collection $\mathcal C$ of convex sets in $\R^d$, if the intersection of any $d+1$ or fewer sets in $\mathcal C$ is nonempty, then $\bigcap \mathcal C$ is nonempty, as well}. Helly's theorem was first proven by Eduard Helly in 1913 (though not published until 1923 \cite{Helly}) and independently by Radon \cite{Radon} in 1921. Helly's theorem was the invitation to a veritable cornucopia of results in discrete geometry, and the last century has led to dozens of variations on Helly's theorem. B\'ar\'any and Kalai's survey \cite{Barany-Helly-survey} provides a recent and extensive look at results in the broad area of related results, while the survey by Amenta, de Loera, and Sober\'on \cite{Helly-survey} focuses more specifically on Helly's theorem.

In 1973, Doignon proved a Helly-type theorem for the integer lattice \cite{Doignon}.

\begin{thm*}[Doignon's theorem]
    Let $\mathcal C$ be a finite collection of convex sets in $\R^d$. If the intersection of any $2^d$ or fewer sets in $\mathcal C$ contains an integer point, then $\bigcap \mathcal C$ does, too.
\end{thm*}

This theorem was also independently discovered by Bell \cite{Bell1977} and Scarf \cite{Scarf1977}. A few years later, Hoffman extended Bell's result by connecting Helly-type theorems with polytopes.

Given a set $S\subseteq\mathbb{R}^d,$ the $\emph{Helly number}$ of $S,$ denoted $h(S),$ is the smallest $h$ such that the following $\emph{Helly-type theorem}$ holds:
\begin{quote}
    Let $\mathcal{F}$ be a finite family of convex sets in $\mathbb{R}^d.$ If every $h$ or fewer sets in $\mathcal{F}$ contains a point of $S$ in their intersection, then the intersection of all sets in $\mathcal{F}$ contains a point in $S.$
\end{quote}
If no such $h$ exists, we say $h(S)=\infty.$ Helly's theorem says that $h(\R^d) \leq d+1$ (while the collection of facets of a simplex shows that $h(\R^d) \geq d+1)$, and Doignon's theorem says that $h(\Z^d) \leq 2^d$.

Moreover, a subset $T\subseteq S$ is called  \emph{empty in $S$} if $\conv(T)\cap S$ only contains vertices of $\conv(T)$. We call a set $S$ \emph{discrete} if it has no accumulation points.

\begin{thm*}[Hoffman \cite{Hoffman1979}]
    If $S\in\mathbb{R}^d$ and $S$ is discrete, then $h(S)$ is equal to the maximum size of an empty subset of $S$.
\end{thm*}

Hoffman actually proved a version of this result that holds for any subset $S \subseteq \R^d$, but we won't need that more general result. Since the unit hypercube $\{0,1\}^d$ is empty in $\Z^d$, we have $h(\Z^d) \geq 2^d$; this shows that Doignon's theorem is tight. Further works on Helly numbers have considered mixed integer subsets $\Z^r \times \R^{d-r}$ \cite{averkov-mixed-integer}, algebraic sets \cite{Helly-survey}, or unions of lattices \cite{garber2023helly}.

In the first part of this paper, we study Helly numbers of product sets, or sets of the form $S = A^d$, for some $A \subseteq \R$. Dillon \cite{Dillon} proved a general criterion on product sets which implies in particular that if $p$ is a polynomial with degree at least $2$ and $A=\{p(n):n\in\mathbb{Z}\},$ then $h(A^2)=\infty$ for all $d\ge 2.$ 
Based on this result, one might conjecture that the finiteness of $h(A^2)$ is related to the sparseness of the set $A$. However, this is ruled out in \cite{Dillon} by the construction of a set $A\subseteq\mathbb{Z}$ whose consecutive elements differ by at most $2$ (a ``$2$-syndetic'' set) such that the $A^2$ has arbitrarily large empty polygons, and thus arbitrarily large Helly number. In Section \ref{infEmpty}, we strengthen this result by constructing a $2$-syndetic set containing an infinite set of vertices in convex position whose convex hull is empty.

Dillon's method, however, gives no information for \emph{exponential lattices}: sets of the form $L_d(\alpha)=\{\alpha^n: n\in\mathbb{N}_0\}^d$. Ambrus, Balko, Frankl, Jung, and Naszódi \cite{Ambrus} studied these sets in two dimensions, and they obtained upper bounds for all $\alpha>1$ and exact values for $\alpha\in[\tfrac{1+\sqrt{5}}{2},\infty).$
\begin{thm*}[Ambrus, Balko, Frankl, Jung, Naszódi \cite{Ambrus}]
Let $\alpha>1.$
    \begin{itemize}
        \item If $\alpha\ge 2,$ then $h(L_2(\alpha))= 5.$
        \item If $\alpha\in[\tfrac{1+\sqrt{5}}{2},2),$ then $h(L_2(\alpha))= 7.$
        \item If $\alpha\in(1,\tfrac{1+\sqrt{5}}{2}),$ then $h(L_2(\alpha))\le 3\lceil\log_{\alpha}(\tfrac{\alpha}{\alpha-1})\rceil+3.$
    \end{itemize}
\end{thm*}
In Section \ref{explatUpper}, we obtain a stronger bound, with an approach that is shorter and more geometric:
\begin{theorem}
\label{expLatStronger}\vspace{-0.5\baselineskip}
For $\alpha>1,$ \[h(L_2(\alpha))\le 2\left\lceil\log_{\alpha}\left(\frac{\alpha}{\alpha-1}\right)\right\rceil+3.\]
\end{theorem}
This inequality recovers the exact values of $h(L_2(\alpha))$ for $\alpha>\tfrac{1+\sqrt{5}}{2}.$ Our method of proof additionally allows us to fully characterize all maximal empty polygons when $\alpha\ge 2.$ 

Ambrus et al.~\cite{Ambrus} also proved a lower bound for $h(L_2(\alpha))$ by constructing an empty polygon with vertices on a hyperbola. They showed that 
for $\alpha>1,$ we have $h(L_2(\alpha))\ge\sqrt{\frac{1}{\alpha-1}}$. In Section~\ref{explatLower}, we extend this bound to all dimensions:
\begin{theorem}
    \label{expLatticeBound3d}
    For $\alpha>1$ and $d>1,$\vspace{-0.25\baselineskip}
    \[h(L_d(\alpha))\ge\binom{k+d-1}{d-1},\vspace{-0.25\baselineskip}\]
    where $k=\left\lfloor\sqrt{\tfrac{1}{\alpha-1}}\,\right\rfloor.$
\end{theorem}

This result is nontrivial when $\alpha<2,$ and is strongest when $\alpha$ is close to $1.$

In Section~\ref{acong}, we investigate Helly numbers of powers of arithmetic congruence sets, which are sets of the form $A= S + m\Z$ for some subset $S \subseteq \{0,1,2\dots,m\}$. This problem falls under the scope of two previous results:
\begin{itemize}
    \item De Loera, La Haye, Oliveros, and Rold\'an-Pensado \cite{DeLoera} showed that $h(S)\leq 3^{3\ell} 2^d$ if $S=\mathbb{Z}\backslash (L_1\cup L_2\cup\dots L_{\ell})$, where $L_1,\dots,L_{\ell}$ are sublattices of $\mathbb{Z}^d.$ If $A = S + m\Z$, then $A^d$ is formed by removing exactly $m^d-|S|^d$ sublattices of $\mathbb{Z}^d$, so $h(A^d)\leq 3^{3(m^d - |S|^d)} 2^d.$
    \item Garber \cite{garber2023helly} showed that $h(S) \leq k2^d$ if $S$ is the union of $k$ translates of $\Z^d$; if $d=2$, he proved that $h(S)\leq k+6$. If $A = S + m\Z$, then $A^d$ is the union of $|S|^d$ translates of $\Z^d$, so we conclude that $h(A^2) \leq |S|^2 + 6$ and $h(A^d) \leq \big(2|S|\big)^d$.
\end{itemize}  
We improve these bounds for certain arithmetic congruence sets.

\begin{restatable}{theorem}{ThmCongruenceSet}\label{aconggeneral}
    Let $m$ be a positive integer with smallest prime factor $p$ and let $k$ and $d$ be positive integers such that $d<\tfrac{p(m-1)}{m(m-k)}.$ If $A = S + m\Z$ where $|S|=k,$ then
    \[
        h(A^d) \le k^d.
    \]
\end{restatable}

This theorem is strongest when $m$ is itself prime, in which case the inequality reduces to $d < (m-1)/(m-k)$, or (after rearranging) $\tfrac{k}{m} > 1-\tfrac{1}{d}(1-\tfrac{1}{m})$. This is stated more cleanly in the following corollary.
\begin{corollary}
    For any prime $p$ and any $\alpha > 1 - \frac{1}{d}\big(1-\frac{1}{p}\big)$, and for any set $S \subseteq \{1,2,\dots,p\}$ with $|S| \geq \alpha p$, define $A = S + p\Z$. In this situation,
    \[
        h(A^d) \leq |S|^d.
    \]
\end{corollary}

We conclude in \Cref{sec:problems} with several open problems.

\section{Exponential lattices}
\subsection{Upper Bound}\label{explatUpper}
Ambrus, Balko, Frankl, Jung, and Nasz\'odi \cite{Ambrus} establish an upper bound on $h\big(L_2(\alpha)\big)$ by considering an empty polygon and dividing its edges into four types. Then they bound the number of edges in each type by making use of the emptiness property.

We also consider an empty polygon in an exponential lattice, and in one case (where all the edge slopes are nonnegative), we repeat the analysis in \cite{Ambrus}. Our remaining analysis, however, considers a different set of points, which is how we achieve an improved bound with reduced casework.

\begin{lemma}[Lemma 10 and Corollary 11 in \cite{Ambrus}]\label{thm:explat-nonnegative-slope}
    An empty polygon in $L_2(\alpha)$ that contains only edges with nonnegative slope has at most $2\lceil\log_{\alpha}(\frac{\alpha}{\alpha-1})\rceil+2$ edges.
\end{lemma}

\begin{proof}[Proof of Theorem~\ref{expLatStronger}]
    Consider an empty polygon $\mathcal{P}$ in $L_2(\alpha).$ If all edges of $\mathcal{P}$ have nonnegative slope, then \Cref{thm:explat-nonnegative-slope} is sufficient.
    
    Otherwise, $\mathcal{P}$ contains vertices $a=(\alpha^p,\alpha^q)$ and $b=(\alpha^r,\alpha^s)$ such that $p<r$ and $q>s.$ We may choose $a$ and $b$ so that no vertices are strictly above and to the left of $a$ or strictly below and to the right of $b$. A diagram is shown in Figure~\ref{fig:explat}.
    
    Let $c=(\alpha^{r-1},\alpha^{q-1})$ and let $\ell_1,\ell_2$ be the horizontal and vertical lines passing through $c$ respectively. Let $x$ and $y$ be the intersections of $\overline{ab}$ with $\ell_1$ and $\ell_2$ respectively. We define $i,j \in \mathbb{N}$ such that $d:=(\alpha^{r-1-i},\alpha^{q-1})$ is the rightmost point of $\ell_1 \cap L_2(\alpha)$ strictly to the left of $\overline{ab},$ and $e:=(\alpha^{r-1},\alpha^{q-1-j})$ is the uppermost point of $\ell_2 \cap L_2(\alpha)$ strictly below $\overline{ab}$. It is possible for $c$ to lie below $\overline{ab},$ in which case the points $c,d,e$ coincide; this will not affect our proof.
    
    \begin{figure}
        \centering
        \begin{tikzpicture}[scale=0.9]
            \foreach \n [evaluate=\n as \x using {1.3^\n}] in {1,2,...,8}
                \draw[gray] (\x,0.5) -- (\x,1.3^7);
            \foreach \n [evaluate=\n as \y using {1.3^\n}] in {1,2,...,7}
                \draw[gray] (0.5,\y) -- (1.3^8,\y);
            \foreach \name/\pos/\m/\n in {b/right/8/4, h/right/8/5, c/above right/7/6, e/below right/7/5, f/below left/5/5, d/above left/5/6}
                \draw[fill=black] ({1.3^\m},{1.3^\n}) circle[radius=1.5pt] node[name=\name,inner sep=0pt] {} node[\pos]{$\name$};
            \draw[fill=black] ({1.3^2},{1.3^7}) circle[radius=1.5pt] node[name=a,inner sep=0pt] {} node[above=2pt]{$a$};
            \draw[fill=black] ({1.3^3},{1.3^7}) circle[radius=1.5pt] node[name=g,inner sep=0pt] {} node[above=0.3pt]{$g$};
            \draw (a) -- (b);
            \draw[fill=darkgray!80] ($(b)!{(1.3^8-1.3^7)/(1.3^8-1.3^2)}!(a)$) circle[radius=1.5pt] node[name=y,inner sep=0pt] {} node[above right, darkgray!80]{$y$};
            \draw[fill=darkgray!80] ($(a)!{(1.3^7-1.3^6)/(1.3^7-1.3^4)}!(b)$) circle[radius=1.5pt] node[name=x,inner sep=0pt] {} node[above, darkgray!80]{$x$};
            \draw[fill=darkgray!80] (0.5,1.3^7) circle[radius=1.5pt] node[name=a',inner sep=0pt] {} node[above=2pt, darkgray!80]{$a'$};
            \draw[fill=darkgray!80] (1.3^8,0.5) circle[radius=1.5pt] node[name=b',inner sep=0pt] {} node[right, darkgray!80]{$b'$};
            \node[left] at (0.5,1.3^6) {$\ell_1$};
            \node[below] at (1.3^7,0.5) {$\ell_2$};   
            \draw[dotted] (a') -- (b');
            \draw[fill=darkgray!80] ($(a')!{(1.3^7-1.3^6)/(1.3^7-0.5)}!(b')$) circle[radius=1.5pt] node[name=x',inner sep=0pt] {} node[above,darkgray!80]{$x\rlap{$'$}$};
            \draw[fill=darkgray!80] ($(b')!{(1.3^8-1.3^7)/(1.3^8-0.5)}!(a')$) circle[radius=1.5pt] node[name=y',inner sep=0pt] {} node[right,darkgray!80]{$y'$};
        \end{tikzpicture}
        \caption{The named points in the proof of Theorem~\ref{expLatStronger}.}
        \label{fig:explat}
    \end{figure}
    To bound $i$ and $j,$ let $a'=(0,\alpha^q)$ and $b'=(\alpha^r,0).$ The lines $\overline{a'b'}$ and $\ell_1$ intersect at $x'=(\alpha^r-\alpha^{r-1},\alpha^{q-1}).$ Since $x'$ is strictly to the left of $x$ and the first coordinate of $x$ is less than $\alpha^{r-i}$, we have $\alpha^{r}-\alpha^{r-1}<\alpha^{r-i}$; in other words,
    \[
    i<\log_{\alpha}\left(\frac{\alpha}{\alpha-1}\right).
    \]
    The same bound holds for $j$ by considering the intersection of $\overline{a'b'}$ and $\ell_2.$
    
    Our next goal is to deduce restrictions on the possible locations of the vertices of $\mathcal P$. Let $f$ be the point such that $cde\!f$ is a rectangle. Let $g$ be the lattice point immediately to the right of $a,$ and let $h$ be the lattice point immediately above $b.$ Consider a point $p$ below $\overline{ab}$ lying outside pentagon $xye\!f\!d.$ 
    \begin{itemize}
        \item If $p$ lies non-strictly below and to the left of $d,$ then triangle $pab$ contains $d.$
        \item If $p$ lies non-strictly below and to the left of $e,$ then triangle $pab$ contains $e.$
        \item Otherwise, $p$ lies strictly below and to the right of $b$ or strictly above and to the left of $a,$ violating our assumption.
    \end{itemize}
    Therefore, all vertices of $\mathcal{P}$ below $\overline{ab}$ must lie within $xye\!f\!d.$
    
    Similarly, consider a point $p$ above $\overline{ab}$ lying outside triangle $cxy.$ 
    \begin{itemize}
        \item If $p$ lies on or above $\ell_1,$ then triangle $pab$ contains $g.$
        \item If $p$ lies on or to the the right of $\ell_2,$ then triangle $pab$ contains $h.$
    \end{itemize}
    Thus, all vertices of $\mathcal{P}$ above $\overline{ab}$ must either be $g$ or $h,$ or lie within $cxy.$
    
    To finish, we bound the number of vertices in rectangle $cdf\!e.$ For any given $x$-coordinate, no two vertices with that $x$-coordinate can lie on the same side of $\overline{ab}$ (and similarly for $y$-coordinates). This implies that the number of vertices in rectangle $cd\!f\!e$ is at most
    \[
        \min(i,j)+\min(i+1,j+1)\le 2\left\lceil\log_{\alpha}\left(\frac{\alpha}{\alpha-1}\right)\right\rceil-1.
    \]
    Including the vertices $a,b,g,h$ yields the desired bound.    
\end{proof}

Based on this proof, we can easily characterize all maximal empty polygons for $\alpha\ge 2.$

\begin{corollary}
    If $\alpha\ge 2,$ then all empty pentagons in $L_2(\alpha)$ have vertices of the form $(\alpha^p,\alpha^q),(\alpha^r,\alpha^s),(\alpha^{r-1},\alpha^{q-1}), (\alpha^{p+1},\alpha^q), (\alpha^{r},\alpha^{s+1}),$ where $p<r$ and $q>s.$
\end{corollary}
\begin{proof}
    \Cref{thm:explat-nonnegative-slope} implies that any empty polygon in which every edge has nonnegative slope has at most 4 vertices. So the pentagon must contain at least one edge with negative slope; we therefore adopt the terminology of our proof of \Cref{expLatStronger}.
    
    If $\alpha\ge 2,$ then $\log_{\alpha}(\tfrac{\alpha}{\alpha-1})\le 1.$ This means both $i$ and $j$ in the above proof are $0,$ which means $c=d=e.$ So the only possible vertices in an empty polygon are $a,b,c,g,h,$ which are exactly the points stated in the corollary. 
\end{proof}

\subsection{Lower Bounds}\label{explatLower}
We will now prove that $h(L_d(\alpha))\geq\binom{k+d-1}{d-1},$ where $k=\left\lfloor\sqrt{\tfrac{1}{\alpha-1}}\,\right\rfloor.$

\begin{proof}[Proof of Theorem~\ref{expLatticeBound3d}]
    Consider the set of points $X$ in $L_d(\alpha)$ which lie on the convex surface $\prod_{i=1}^d x_i = \alpha^k.$  We claim that $\conv(X)$ is empty in $L_d(\alpha)$. Since $\prod_{i=1}^d x_i = \alpha^k$ is a convex function, the only possible points of $L_d(\alpha)$ that lie in $\conv(X)$ satisfy $\prod_{i=1}^d x_i \geq \alpha^k$. All the points satisfying $\prod_{i=1}^d x_i = \alpha^k$ are vertices of $\conv(X)$, so any other point in $\conv(X) \cap L_d(\alpha)$ satisfies $\prod_{i=1}^d x_i \geq \alpha^{k+1}$. On the other hand, the polytope $\conv(X)$ lies completely in the half-space $\sum_{i=1}^d x_i \leq \alpha^k + d - 1$. (The bounding hyperplane of this half-space coincides with the facet determined by the vertices $(1,1,\ldots,\alpha^k),\ldots,(\alpha^k,1,\ldots,1)$.) Therefore, to show that $\conv(X)$ is empty in $L_d(\alpha)$, it suffices to show that the surface $\prod_{i=1}^d x_i \geq \alpha^{k+1}$ lies strictly above the hyperplane $\sum_{i=1}^d x_i = \alpha^k + d - 1$.
    
    The point on $\prod_{i=1}^d x_i \geq \alpha^{k+1}$ with minimum sum of coordinates is $(\alpha^{(k+1)/d},\ldots,\alpha^{(k+1)/d})$, so it suffices to show that
    \[d\alpha^{(k+1)/d}\ge\alpha^k+d-1.\]
    We will apply the substitution $\alpha=1+s^2$. Using the inequalities $(1+s^2)^{(k+1)/d}\ge 1+\tfrac{k+1}{d}s^2$ (by the binomial theorem) and $(1+s^2)^k<e^{s^2k}$ (by Taylor approximation), we can see that it is enough to prove that
    \[d\left(1+\frac{k+1}{d}s^2\right)\ge e^{s^2k}+d-1.\]
    Our particular choice of $k$ is $k=1/s$, so the previous inequality simplifies to 
    \[s^2+s+1\ge e^s.\]
    This is true because $k\geq 1$ and thus $s\leq 1.$ So we have found an empty polytope. 
    
    The number of vertices in this polytope is the number of ordered $d$-tuples of nonnegative integers $(n_1,n_2,\dots,n_d)$ with $\prod_{i=1}^k \alpha^{n_i} = \alpha^k$; in other words, the number of $d$-tuples $(n_1,\dots,n_d)$ with $\sum_{i=1}^d n_i = k$. This is well known to be $\binom{k+d-1}{d-1}$ (for example, by the ``stars and bars'' method).
\end{proof}

\section{2-Syndetic Sets}\label{infEmpty}

A set $A \subseteq \Z$ is called \emph{$2$-syndetic} if for every $n \in \Z$, either $n \in A$ or $n+1 \in A$ (or both). To construct a $2$-syndetic set $A$ for which $A^2$ has infinite Helly number, Dillon \cite{Dillon} constructed a sequence of successively larger empty polygons using a sequence of rational approximation to a fixed irrational number. In \cite{Ambrus}, Ambrus, \textit{et al.} prove that the ``Fibonacci lattice'' contains empty polygons with arbitrarily many vertices. We combine these approaches to construct a 2-syndetic set containing an infinite set of vertices in convex position whose convex hull is empty.

\begin{proposition}
    There is a set $A \subseteq \Z$ such that
    \begin{itemize}
        \item for every $n \in \Z$, either $n \in A$ or $n+1 \in A$ (or both), and
        \item $A\times A$ contains an infinite set $\{p_i\}_{i\geq 1}$ in convex position such that $\conv{\{p_i\}_{i\geq 1}}$ is empty.
    \end{itemize}
\end{proposition}
\begin{proof}
Define the Fibonacci numbers by $F_0=0, F_1=1,$ and $F_n=F_{n-1}+F_{n-2}.$ Let $\phi=\tfrac{1+\sqrt{5}}{2}$ and $\psi=\tfrac{1-\sqrt{5}}{2}.$ Set
\[\{p_i=(-F_{2i},2F_{2i+1}-1)\}_{i\ge 1}.\]
Using Binet's formula, the slope of $\overline{p_ip_{i+1}}$ is 
\begin{align*}\frac{(2F_{2i+3}-1)-(2F_{2i+1}-1)}{(-F_{2i+2})-(-F_{2i})}&=-2\frac{F_{2i+2}}{F_{2i+1}}\\&=-2\left(\frac{\phi^{2i+2}-\psi^{2i+2}}{\phi^{2i+1}-\psi^{2i+2}}\right)\\&=-2\left(\phi+\frac{(\phi-\psi)(\psi^{2i+1})}{\phi^{2i+1}-\psi^{2i+1}}\right)\\&=-2\left(\phi+\frac{\psi^{2i+1}}{F_{2i+1}}\right).\end{align*}
Because $-1<\psi < 0$ and $F_{2i+1}$ is increasing, this expression decreases as $i$ increases. So the points $p_i$ are in convex position. Let $\mathcal{P}=\conv{\{p_i\}_{i\geq 1}}.$

Let $\ell$ be the line $y=-2\phi x.$ By Binet's Formula, the vertical distance between $p_i$ and $\ell$ is

\[2\phi\left(\frac{\phi^{2i}-\psi^{2i}}{\phi-\psi}\right)-\left(\frac{2\phi^{2i+1}-2\psi^{2i+1}}{\phi-\psi}-1\right)=1-2\psi^{2i}.\]
In particular, as $i$ increases, the distance between $p_i$ and $\ell$ increases and approaches $1.$

Let $L$ be the strip defined by $-2\phi x - 1 < y = -2\phi x$. We have concluded that $\mathcal P \subseteq L$. Since the vertical width of $L$ is 1, no two integer points in $L$ have the same $y$ coordinate; because the slope of $\ell$ is less than $-2$, the $y$-coordinates of any two lattice points in $\mathcal P$ differ by at least 2. In particular, the $y$-coordinates of the points $p_i$ are distinct from the $y$-coordinates of any other integer points that $\mathcal P$ contains. If $Y$ be the set of $y$-coordinates of non-vertex lattice points in $\mathcal P$. Then $A = \Z \setminus Y$ is a 2-syndetic set and $\mathcal P$ is a empty in $A\times A,$ as desired.
\end{proof}

\section{Arithmetic Congruence Sets}\label{acong}

The main result of this section is an upper bound on arithmetic congruence sets, restated here for convenience.

\ThmCongruenceSet*

\begin{proof}
Consider an empty polygon $\mathcal{P}$ with vertices in $A^d.$ Suppose for the sake of contradiction that $\mathcal{P}$ has two vertices $u$ and $v$ with $v-u \in m\Z^d$. For each $1 \leq j \leq m-1$, let $c_j = u + \frac{j}{m} (v-u)$, which is an integer point. Fix $1\leq i\leq d$, and consider the $i$th coordinates of the points $c_1,\dots,c_{m-1},$ which form an arithmetic progression modulo $m.$
\begin{itemize}
    \item If this arithmetic progression is constant modulo $m,$ then since the $i$th coordinates of $u$ and $v$ are in $A,$ the $i$th coordinates of $c_1,\dots,c_{m-1}$ are also in the set $A.$
    \item If this arithmetic progression is nonconstant, then any residue appears in the progression at most $\tfrac{m}{p}$ times. Therefore, the $i$th coordinates of $c_1,\dots,c_{m-1}$ contain at most $\tfrac{m}{p}(m-k)$ residues \textbf{not} in $A.$
\end{itemize}
In either case, the $i$th coordinates of $c_1,\dots,c_{m-1}$ contain at most $\tfrac{m}{p}(m-k)$ residues not in $A$. Summing over all $i,$ there are at most $\tfrac{m}{p}(m-k)d$ coordinates among $c_1,\dots,c_{m-1}$ which are not in $A.$ However, we know $\tfrac{m}{p}(m-k)d<m-1.$ So by the Pigeonhole Principle, there is some $c_i$ with all of its coordinates in $A.$ This violates the emptiness condition.

Thus, there is at most one vertex of any given sequence of remainders modulo $m,$ so there are at most $k^d$ total vertices in $\mathcal{P}.$
\end{proof}

The assumption that $d<\tfrac{p(m-1)}{m(m-k)}$ cannot be strengthened to an equality. For example, the following case satisfies $d=\tfrac{p(m-1)}{m(m-k)}$ but $h(A^d)>k^d.$

\begin{proposition}\label{thm:01mod3}
     if $S = \{0,1\}$ and $A = S + 3\Z$, then $h(A^2) = 8$.
\end{proposition}
\begin{proof}
A construction of $8$ points is shown in Figure~\ref{8const}.

\begin{figure}[b]
\centering
\begin{tikzpicture}[gray, scale=0.7]
    \draw (0,-1)--(0,5);
    \draw (1,-1)--(1,5);
    \draw (3,-1)--(3,5);
    \draw (4,-1)--(4,5);
    \draw (6,-1)--(6,5);
    \draw (7,-1)--(7,5);
    \draw (-1,0)--(8,0);
    \draw (-1,1)--(8,1);
    \draw (-1,3)--(8,3);
    \draw (-1,4)--(8,4);
    \draw[ForestGreen, thick, every node/.style={ForestGreen,shape=circle,fill,inner sep=1.4pt}] (0,0) node{}--(1,0) node{}--(3,1) node{}--(6,3) node{}--(7,4) node{}--(6,4) node{}--(4,3) node{}--(1,1) node{}--cycle;
\end{tikzpicture}
\caption{An $8$-vertex empty polygon in $\big(\{0,1\}+3\Z\mspace{2mu}\big)^2$.}
\label{8const}
\end{figure}
We claim that any empty polygon $P$ has at most two vertices whose modulo-$3$ residues are $(0,0)$. Then, by applying the symmetry of $A$, each of the four possible residues in $(\Z/3\Z)^2$ can appear at most twice among the vertices of an empty polygon---so any empty polygon has at most $8$ vertices.

Suppose for the sake of contradiction that $(3x_1,3y_1),(3x_2,3y_2),(3x_3,3y_3)$ are all vertices of $P$.
Since $P$ is empty, the centroid $(x_1+x_2+x_3,y_1+y_2+y_3)$ must have a coordinate that is congruent to $2$ (mod 3). We may assume $x_1+x_2+x_3\equiv 2\pmod{3}.$ Since the residues of $x_1,x_2,x_3$ take values in $\{0,1\}$, they cannot be distinct modulo $3$. The arguments are all the same, so assume $x_1\equiv x_2\pmod{3}.$ The congruence conditions guarantee that following convex combination are integer points:
\[(2x_1+x_2,2y_1+y_2),\]
\[(2x_2+x_1,2y_2+y_1).\]
Since the $x$-coordinates of both points are 0 (mod 3), the $y$-coordinates of both points must have remainder 2 (mod 3). But this is impossible, since the sum of the $y$-coordinates is $0\pmod{3}$. Therefore $P$ contains at most 2 vertices congruent to $(0,0)$ mod $3\Z^2$. 
\end{proof}

\section{Open Problems}\label{sec:problems}
The main unsolved problem in  this paper is bounding the number of vertices of empty polygons in exponential lattices in 3 or more dimensions. In particular, it is still unknown whether any three-dimensional exponential lattice contains an empty polygon with infinitely many vertices.

\begin{problem}
    Is $h\big(L_3(\alpha)\big) < \infty$?
\end{problem}

While the majority of our proof of \Cref{expLatStronger} does extend to higher dimensions, we cite \Cref{thm:explat-nonnegative-slope} to handle a special case, and its proof in \cite{Ambrus} relies on ordering the edge slopes of a polygon. Since facets in higher dimensions cannot be ordered in this way, a new method is required.

The growth of $h(L_d(\alpha))$---assuming, of course, it is finite---is also unknown. In the two-dimensional case, Ambrus, Balko, Frankl, Jung, and Nasz\'odi~\cite{Ambrus} note that if $\alpha=1+\tfrac{1}{x},$ the best known bounds in the plane are
\[
    \lfloor\sqrt{x}\rfloor\le h(L_2(\alpha))\le C\mspace{2mu} x\log(x)
\]
for some $C > 0$.

\begin{problem}
    Improve the asymptotic bounds on $h\big( L_2(1+\frac{1}{x})\big)$ as $x \to \infty$.
\end{problem}

In higher dimensions, as $x\to\infty,$ \Cref{explatLower} shows that for every $d$, there is a $c_d > 0$ such that
\[
    c_d\, x^{(d-1)/2}\le h(L_d(\alpha)),
\]
though of course we have no corresponding upper bound.

Our bounds on arithmetic congruence sets are strongest when $m$ is prime and $S$ is a significant proportion of $\{1,2,\dots,m\}$. What happens when these conditions aren't true? There may be much to explore in this direction.

\begin{problem}
    Prove further bounds on Helly numbers for arithmetic congruence sets.
\end{problem}

\section{Acknowledgements}
This research was conducted under the auspices of the \scalebox{0.9}{MIT} \scalebox{0.9}{PRIMES}--\scalebox{0.9}{USA} program. Dillon was further supported by a National Science Foundation Graduate Research Fellowship under Grant No. 2141064.

\bibliographystyle{amsplain-nodash}
\bibliography{ref}

\end{document}